\documentclass{amsart}
\usepackage{amsfonts}
\usepackage{pgfplots}
\usepackage{amsmath}
\usepackage{calc}
\usepackage{amsthm}
\usepackage{amssymb}
\usepackage{hyperref}
\usepackage{autonum}
\usepackage[utf8]{inputenc}
\usepackage{relsize}
\usepackage[T1]{fontenc}
\usepackage{soul}

\usepackage{tikz}
\usepackage{tikz-cd}

\pgfplotsset{compat=newest}

\newtheorem{theorem}{Theorem}
\newtheorem*{theorem*}{Theorem}

\newtheorem{proposition}[theorem]{Proposition}
\newtheorem{lemma}[theorem]{Lemma}
\newtheorem{definition}[theorem]{Definition}

\newtheorem{setup}[theorem]{Setup}
\newtheorem{example}[theorem]{Example}

\newtheorem{remark}[theorem]{Remark}

\newcommand{\calO}{\mathcal{O}}
\newcommand{\OK}{\mathcal{O}_K}
\newcommand{\IC}{\mathbb{C}}
\newcommand{\KR}{K_\mathbb{R}}
\newcommand{\R}{\mathbb{R}}
\newcommand{\Q}{\mathbb{Q}}

\DeclareMathOperator{\SL}{SL}
\DeclareMathOperator{\ML}{Mod}
\DeclareMathOperator{\GL}{GL}

\DeclareMathOperator{\cl}{cl}

\DeclareMathOperator{\vol}{vol}

\DeclareMathOperator{\Err}{Err}

\DeclareMathOperator{\card}{\texttt{\#}}

\DeclareMathOperator{\N}{N}

\begin{document}
\title{Module lattices and their shortest vectors}
\author[N. Gargava]{Nihar Gargava}
\author[V. Serban]{Vlad Serban}
\author[M. Viazovska]{Maryna Viazovska}
\author[I. Viglino]{Ilaria Viglino}

\address{N. Gargava, Section of Mathematics, EPFL, Switzerland}
\email{nihar.gargava@epfl.ch}
\address{V. Serban, Department of Mathematics, New College of Florida, U.S.A.}
\email{vserban@ncf.edu}
\address{M. Viazovska, Section of Mathematics, EPFL, Switzerland}
\email{maryna.viazovska@epfl.ch}
\address{I. Viglino, Section of Mathematics, EPFL, Switzerland}
\email{ilaria.viglino@epfl.ch}

\begin{abstract}
We study the shortest vector lengths in module lattices over arbitrary number fields, with an emphasis on cyclotomic fields. In particular, we sharpen the techniques of \cite{GSV23} to establish improved results for the variance of the number of lattice vectors of bounded Euclidean norm in a random module lattice. We then derive tight probabilistic bounds for the shortest vector lengths for several notions of random module lattice. 
\end{abstract}
\maketitle

\section{Introduction}

One of the central concerns in lattice-based cryptography is to prove or disprove the existence
of fast algorithms that can find or approximate the length of the shortest vectors of a lattice.
The non-existence of such algorithms is still a wide-open question, and there exist well-publicized open challenges \cite{BLR08}
inviting competitors to submit algorithms solving the shortest vector problem (SVP) for "randomly" produced lattices. For random lattices $\Lambda$ in increasing dimension $n$, it turns out that the length of the shortest vector, denoted by $\lambda_1(\Lambda)$, can be predicted almost exactly:
\begin{theorem}
\cite[Theorem 5]{LN20}
\label{th:phong}
Let $\Lambda\subseteq \mathbb{R}^{n}$ be a Haar-random lattice chosen from the space $\SL_{n}(\mathbb{R})/\SL_{n}(\mathbb{Z})$. Then, as $n \rightarrow \infty$, with probability $1- o(1)$ we have 
\begin{equation}
  1-\frac{\log \log n}{n} \leq 2^{-\frac{1}{n}}\cdot\frac{\lambda_1(\Lambda)}{\gamma(n)} \leq 1 + \frac{\log\log n}{n},
\end{equation}
where $\gamma(n)$ is radius of a ball of unit volume in $\R^n$. 
\end{theorem}
\begin{remark}
\label{re:gamma}
For large $n$, $\gamma(n) \simeq \sqrt{\tfrac{n}{2 \pi e}}$ by Stirling's approximation.
\end{remark}

A weaker form of this is stated by Ajtai in \cite{A02} without proof.
In the lattice challenges in \cite{BLR08}, this result is used to set benchmarks for challenge contenders. Such results greatly inform our understanding and analysis of lattice reduction algorithms \cite{LN20,EK20}. \par

Ever since the arrival of the NTRU cryptosystem \cite{NTRUpaper}, lattices with additional algebraic structure called module lattices, which do not exhibit the same behavior as the Haar random lattices of Theorem \ref{th:phong}, have been the object of study and scrutiny as they promise efficiency gains compared to unstructured lattices. Module lattices now underlie the security of several NIST post-quantum standards \cite{NIST}. It is therefore crucial to understand and quantify the behavior of short vectors as well as the hardness of computational problems such as SVP on such lattices.
The lattices in question are $\OK$-modules, where $\OK$ is the ring of integers of a number field $K$, and the latter is often taken to be a cyclotomic field. These are one of the main special classes of lattices studied in the context of post-quantum cryptography \cite{LS15,DMPT23}. See also work on ring-based versions of the learning with errors (LWE) and short integer solutions (SIS) problems such as \cite{RingLWE2010,RingLWE}.

Producing asymptotic estimates like Theorem \ref{th:phong} for module lattices in growing dimensions is complicated by the additional arithmetic complexity of the objects. In addition, in this setting, one ideally also wants to understand the behavior when increasing the degree of the number field $K$, leading to further increasing complexity. In this article, we obtain results analogous to the above theorem for many families of module lattices. For example, we obtain:
\begin{theorem}
	\label{th:as_module_boundintro}
Let $t \geq 11$ be fixed.
Let $K = \mathbb{Q}(\zeta_k)$ be a cyclotomic number field and let $n = t \cdot \deg K = t \varphi(k)$. Let $\omega_K=\card\mu(K)$ which is $k$ or $2k$ depending on whether $k$ is even or odd. Consider a Haar-random 
unit covolume module lattice $\Lambda \subset K^{t} \otimes \mathbb{R}\cong \R^n$ in the moduli space of rank-$t$ module lattices over $K$.
Then, as $k \rightarrow \infty$, we have with probability $1- o(1)$ that
\begin{equation}
 1- \frac{\log \log \omega_K}{n} \leq \omega_K^{-\frac{1}{n}}\cdot\frac{\lambda_1(\Lambda)}{\gamma(n)} \leq  1 + \frac{\log \log \omega_K}{n}.
 \label{eq:bound}
\end{equation}
\end{theorem}
The $t \geq 11$ condition is probably not tight and appears for technical reasons. Improvements to our methods might further reduce this lower bound to values like 3,4,5 (we need $t > 2$ for two moments to exist). Note that in the statement above the factor $\omega_K^\frac{1}{d}  \simeq 1 + \tfrac{\log k}{t\varphi(k)}$ derives from the additional symmetry of module lattices and that the lower bound is valid more generally for $t\geq 2$ (see Remark \ref{rem:unconditionallower}). \par 

What is more, combining our methods with some computation, one can go beyond asymptotic results and we show how to obtain honest probabilistic bounds on the shortest vector for many number fields and $\OK$-ranks that underlie the security of currently implemented schemes. See Example \ref{ex:fixedK} for a case study that required minimal computation and figure \ref{fig:momenterrors} for a host of additional error bounds. We emphasize that these methods produce actual theorems and not just heuristic results. \par

The sequence of number fields chosen in Theorem \ref{th:as_module_boundintro} can also be changed as long as one has uniform bounds on the absolute Weil height of algebraic numbers. This is always true for abelian extensions such as cyclotomic fields, as well as a plethora of other examples and is referred to in the mathematical community as the \emph{Bogomolov property} for the union of all the fields considered, see e.g. \cite[Chapter 11]{MS21} for some results along those lines. \par
The random module lattices in question are studied via a Rogers integration formula \cite{K19,hughes2023mean,GSV23}, leading to moment estimates for the number of non-trivial lattice points of bounded Euclidean length as a random variable over the probability space of Haar-random lattices. Such results are interesting in their own right, and the first three authors in \cite{GSV23} derived asymptotic formulas for the moments under the height assumptions mentioned above. \par
This article, then, has two major objectives: first, we sharpen the second moment results of \cite{GSV23} in Section \ref{sec:moments} with a particular focus on cyclotomic fields, aiming for a slightly more accessible exposition. For example, the minimal rank condition is improved from $t\geq 27$ in \cite{GSV23} to $t\geq 11$ as appearing in Theorems \ref{th:as_module_boundintro} and \ref{thm:cyclotomicnew}. This is mainly achieved by bounding more carefully the contribution to error terms of algebraic numbers $\alpha\in K^\times$ of low Weil height, using specific knowledge at our disposal for cyclotomic fields (see e.g., Theorem \ref{thm:heightexceptions} and Proposition \ref{prop:cycloheightbounds}). \par

Second, we show how to derive shortest vector bounds from the moment results in Section \ref{sec:SVbounds} and carry out a prototypical computation for a particular number field $K=\Q(\zeta_{16})$ and blocksize $256$ in Section \ref{sec:example} to showcase how our results apply to such a concrete situation. See also figure \ref{fig:momenterrors} for bounds for other cyclotomic fields and ranks. 
\par
Finally, we highlight a connection to coding theory: in a companion paper \cite{KatznelsonOK}, it is shown that moments for $q$-ary type lattices obtained by lifting algebraic codes from finite fields to $\OK^t$ (see Construction A in \cite{CS}) converge for large enough parameters to the moments for the full space of module lattices. This crucially implies that results such as Theorem \ref{th:as_module_boundintro} also hold for uniformly random lattices in large enough discretized sets of this form. 

\section*{Acknowledgements}

We would like to thank Andreas Str\"ombergsson for comments on our preprint \cite{GSV23} that informed some aspects of this article. The writing has also benefited from discussions with L\'eo Ducas, Lynn Engleberts, Thomas Espitau, Seungki Kim, Phong Nguyen and Paola de Perthuis. 
This research was partly funded by the Swiss National Science Foundation (SNSF), Project funding (Div. I-III), "Optimal configurations in multidimensional spaces",
184927.

\section{Random module lattices}
We first fix some notations used throughout the text. We shall denote by $K$ a number field (a finite field extension of the rational numbers), by $\deg(K)=[K:\Q]$ or simply $d$ its field degree over the rationals and by $\OK$ its ring of integers. The main example to keep in mind throughout the text is that of a cyclotomic field $K=\mathbb{Q}(\zeta_m)$ obtained by adjoining a primitive $m$-th root of unity $\zeta_m=e^{2\pi i/m}$ to $\Q$. In this case the degree $\deg(K)=\varphi(m)$ is given by Euler's totient function and $\OK=\mathbb{Z}[\zeta_m]$. The invertible elements under multiplication in $\OK$ will be denoted by $\OK^\times$. 
\subsection{Embeddings and unit structure}
Any number field $K$ has $\deg(K)$ distinct field embeddings $\sigma:K\to \mathbb{C}$ which can be grouped into $r_1$ embeddings whose image is contained in $\R$ and $r_2$ pairs of complex conjugate embeddings. The pair $(r_1,r_2)$ is referred to as the signature of $K$ and one has $r_1+2r_2=\deg(K)$. These embeddings naturally map our number field 
$$i_K:K\hookrightarrow K\otimes_{\Q}\R$$
into a $\deg(K)$ dimensional Euclidean space $K\otimes_{\Q}\R\cong \R^{r_1}\otimes \mathbb{C}^{r_2}$ which we will abbreviate by $K_{\R}$. The image $i_K(\OK)$ of $\OK$ via this map is a lattice in $K_{\R}$, and similarly any, say, integral ideal $I\subset \OK$ gives rise to a lattice. For a given positive integer $t$, we will also consider $t$ copies of this construction. In this way, for instance any tuple $(I_1,\ldots,I_t)$ of $\OK$-ideals gives rise to a lattice in $t\cdot \deg(K)$-dimensional Euclidean space $K_{\R}^t$. We shall study the behaviour of short vectors in such lattices. \par
To that end, the structure of the unit group $\OK^\times$ will play a crucial role. Recall that Dirichlet's unit theorem states that this is a free abelian group of rank $r(K)=r_1+r_2-1$, so that, denoting by $\mu(K)$ the roots of unity in $K$, we have the identification as groups:
$$\OK^\times\cong \mu(K)\times \mathbb{Z}^{r_1+r_2-1}.$$
We shall denote the size of the torsion in $\OK^\times$ by $\omega_K=\card \mu(K)$. Note that for $\mathbb{Q}(\zeta_m)$ and $m\neq 2$, all the embeddings are complex, so that $\mathbb{Z}(\zeta_m)^\times$ has rank $\varphi(m)/2-1$ and the number of roots of unity is $\omega_K=2m$ if $m$ is odd and $m$ otherwise.

\subsection{The space of module lattices and components}
Our goal is to study random $\OK$-modules of rank $t>0$ inside Euclidean space $K_{\R}^t$. Given that the questions we consider are invariant under scaling of the lattices, it is preferable and customary to fix the determinant of the lattice and work with spaces of lattices of unit covolume. We take this point of view for the rest of the paper. 
Similarly to the setup of \cite{DK22} we then define: 
\begin{definition}
A module lattice in $\KR^{t}$ is a pair 
$(g,M)$ where $g \in \GL_{t}(K \otimes \mathbb{R})$ and $M \subseteq K^{t}$ is a finitely generated $\OK$-module of maximal rank so that 
\begin{equation}
	\vol\left(K^{t} \otimes \mathbb{R} / (g^{-1} \cdot M ) \right) = 1.
\end{equation}
Two module lattices are considered equal if the lattice $g^{-1}M \subseteq K^{t} \otimes \mathbb{R}$ is identical.

A module lattice is called rational, or simply an $\OK$-module, if it is equal to $(1,M)$ for some $M \subseteq K^{n}$.
When not ambiguous, we may refer to $M$ as the module lattice $(1, M)$. The set of all module lattices of rank $t$ over $K$ shall be denoted by $\ML_{t}(K)$.
\end{definition}
Module lattices inherit a natural action of $\GL_{t}(K_\mathbb{R})$ via multiplication on the left: 
\begin{equation}
 h :(g,M)  \mapsto (  |\N(\det h)|^{-\frac{1}{t \deg K}}h g , M).
\end{equation}
This descends to an action of $\GL_t(\KR)/\mathbb{R}_{>0}$. Furthermore, each space of module lattices comes equipped with a map to the class group $\cl(K)$ of $K$ which picks up the so-called Steinitz class of $(g,M)\in\ML_{t}(K)$: the ideal class of the fractional ideal generated by all the determinants of $t$-tuples of vectors in $M$. It is not hard to see that the space $\ML_{t}(K)$ consists then of one connected component for each Steinitz class on $\cl(K)$ and that each connected component looks like $\GL_t(\KR)/\Gamma \cdot \mathbb{R}_{>0}$ for some arithmetic group $\Gamma$. \par
This produces a natural probability measure on $\ML_{t}(K)$, since each of the components carry a Haar measure as a homogeneous space of $\GL_t(\KR)/\mathbb{R}_{>0}$ which is unique if we prescribe the total mass. Note also that for $t\geq 2$, $\ML_{t}(K)$ is a non-compact probability space, whereas for $t=1$, this is the space of ideal lattices which is compact due to the finiteness of the class group. 
\begin{remark}
\label{re:adelic}

One could take an adelic point of view as in \cite{DK22} and see $\ML_t(K)$ as an adelic quotient. 
In this way, $\ML_t(K)$ really is seen to carry a Haar measure as a homogeneous space of a locally compact group.
\end{remark}
\subsection{Lattices from codes and reduction to principal component}\label{subsec:codes}
Our main object of study in this paper will be, for a fixed origin-centered ball $B$ of volume $V$ in $\KR^t$, the random variable 
\begin{equation}\label{eq:defrhoLambda}
    \rho_V(\Lambda):=\card \left(B\cap \left(\Lambda\setminus \{0\}\right)\right)
\end{equation}
where $\Lambda$ is a random unit covolume module lattice over $K$ in $\deg(K)\cdot t$-dimensional Euclidean space. There are several candidate spaces of such random lattices which one might potentially be interested in: 
\begin{enumerate}
    \item The full space $\ML_t(K)$ as defined above. 
    \item The principal component, whose study reduces to the space $$\SL_t(\KR)/\SL_t(\OK),$$
    which carries a canonical Haar probability measure. 
    \item The discrete sets of lattices which, for projections $\pi_\mathcal{P}:\OK\to\OK/\mathcal{P}=k_\mathcal{P}$, are obtained by lifting algebraic codes over the residue field $k_\mathcal{P}$ as follows: 
    \begin{equation}
\mathcal{L}(\mathcal{P},s) = \{ \beta \pi_\mathcal{P}^{-1}(S) \mid S\subseteq k_\mathcal{P}^{t} \text{ is an $s$-dimensional $k_\mathcal{P}$-subspace}\},
\label{eq:def_of_L}
\end{equation}
where $\mathcal{P}\subset \OK$ is an unramified prime ideal, $1\leq s\leq t-1$ and where the scaling factor 
$$\beta = \beta(\mathcal{P},s) = \N(\mathcal{P})^{-\left(1-\frac{s}{t}\right)\frac{1}{[K:\mathbb{Q}]}}$$
ensures the lattice $ \beta \pi_\mathcal{P}^{-1}(S)$ has the same covolume as $\OK^t$ (and we may normalize so that all of these are one). In particular, as the norm of $\mathcal{P}$ and thus the residue field cardinality $\card k_\mathcal{P}$ increase, the sets $\mathcal{L}(\mathcal{P},s)$ are increasingly large discrete subsets of $\ML_t(K)$.
\end{enumerate}
The first two agree (up to a compact torus) when $K$ has class number one, but in general the second is slightly simpler to work with. The lattices $\mathcal{L}(\mathcal{P},s)$ lie on the same connected component of Steinitz class as the ideal class of $\mathcal{P}^{t-s}$, whereas the ideal classes of primes $\mathcal{P}$ are known to equidistribute in the class group as the norm $\N(\mathcal{P})\to\infty$. We record the following useful remark: 
\begin{remark}
    It suffices to compute the moments of the random variable $\rho_V(\Lambda)$ for random lattices $\Lambda \in\SL_t(\KR)/\SL_t(\OK)$. 
\end{remark}
Indeed, the behavior is the same for each component of $\ML_t(K)$ (see e.g. \cite[Remark 7]{GSV23}). Moreover, we show in \cite{KatznelsonOK} that, under some mild conditions on $s$, the moments over the discrete spaces $\mathcal{L}(\mathcal{P},s)$ converge to the moments for $SL_t(\KR)/\SL_t(\OK)$ as $\N(\mathcal{P})\to\infty$. \par
Henceforth, we therefore focus on evaluating moments for $\rho_V(\Lambda)$ in the principal case, but highlight the fact that this implies such results therefore hold also for the full space of module lattices as well as for lattices in $\mathcal{L}(\mathcal{P},s)$ for large norm of $\mathcal{P}$.

\section{Second moment results}\label{sec:moments}
In this section, we show how to obtain via the the approach devised in \cite{GSV23} improved second moment results for the number of short vectors of bounded Euclidean norm. Our improvements are primarily focused on lattices over cyclotomic fields, although the methods apply more widely. 
\subsection{Additional symmetry and moments}
Fixing an origin-centered ball $B$ of volume $V$, we recall our random variable on the space of lattices of covolume one $\rho_V(\Lambda)$ defined in \eqref{eq:defrhoLambda} and counting non-trivial lattice vectors in $B$. C.L. Siegel's classical mean value theorem \cite{Sie45} implies in particular that the expected value satisfies
$$\mathbb{E}[\rho_V(\Lambda)]=V$$
for unstructured lattices, and it is not hard to show that the same holds for module lattices (see e.g., \cite[Lemma 1]{AV}). C.A. Rogers then developed an integral formula expressing the higher moments of $\rho_V(\Lambda)$ in terms of a rather involved sum over various integral contributions from linear subspaces \cite{Rogers55}. He then proceeded to estimate all of the contributions in the integral formula and showed: 
\begin{theorem}[\cite{R1956}]\label{thm:Rogersmoments}
   The $k$-th moment of $\rho_V(\Lambda)$ for random $\Lambda \in \SL_t(\R)/\SL_t(\mathbb{Z})$ satisfies, provided $t\geq \lceil k^2/4 +3\rceil$, the bound 
   $$0\leq \mathbb{E}[\rho_V(\Lambda)^k]-2^k \cdot m_{k}( \tfrac{V}{2} )\leq V\cdot E_{k,t},$$
   where \begin{equation}
	m_{k}(\lambda) = e^{-\lambda}\sum_{r=0}^\infty\frac{\lambda^{r}}{r!} r^{k}  = \mathbb{E}_{X \sim \mathcal{P}(\lambda) }(X^{k})
\label{eq:def_of_poisson}
\end{equation}
is the $k$-th moment of a Poisson distribution and the error term $E_{k,t}$ decays exponentially in $t$. 
\end{theorem}
In other words, noting that short vectors come in pairs due to the imposed symmetry under $\mathbb{Z}^\times=\{\pm 1\}$, Rogers showed that the number of pairs of short non-zero vectors in $B$ approaches a Poisson distribution of mean $\tfrac{V}{2}$. \par 
The crucial difference for lattices $\Lambda\in \ML_t(K)$ is that units $\OK^\times$ act diagonally (with respect to the $t$ copies of $\KR$) on lattice vectors. In particular, roots of unity in $\OK^\times$ lie on the unit circle in each complex embedding and thus their action preserves Euclidean distances, leading to the important observation: 
\begin{remark}
    For $\Lambda\in \ML_t(K)$, the random value $\rho_V(\Lambda)$ is a multiple of the number of roots of unity $\omega_K$. 
\end{remark}
Therefore, the natural object to study is $\omega_K$-tuples of vectors in $B\cap \Lambda$. In light of both the mean value results and Theorem \ref{thm:Rogersmoments}, one might therefore expect the number of $\omega_K$-tuples to converge to a Poisson distribution of mean $\tfrac{V}{\omega_K}$ as the dimension increases. This was indeed shown in \cite{GSV23} in quite some generality. In this case, the analogues of Rogers' integral formula \cite{K19,hughes2023mean,GSV23} are still available, but the technical obstacles is in estimating the contributions of the various subspaces. Roughly, the additional unit action's impact can be seen in $\omega_K$ replacing the role of $2$ in the main moment terms, but the presence of typically infinitely many non-torsion units in $\OK^\times$ makes the error terms considerably more delicate to estimate. 
\par 
For the rest of this section, we give an exposition of a sharpening of the techniques in \cite{GSV23} in the second moment case. Our goal is therefore to show that, for $\Lambda\in\SL_t(\KR)/\SL_t(\OK)$ and $\OK$-rank $t$ large enough, we have the bounds 
\begin{equation}\label{eq:toshowmoment}
0\leq \mathbb{E}[\rho_V(\Lambda)^2]-(V^2+\omega_K\cdot V)\leq V\cdot  E_{K,t},
\end{equation}
for some error term $E_{K,t}$ that decays rapidly as the lattice dimension increases. 
\subsection{Integral formula and heights}
The Rogers integral formula for module lattices implies for a fixed origin-centered ball $B$ in $t$ copies of Euclidean space $K\otimes_\mathbb{Q}\mathbb{R}$ that the second moment equals: 
$$\mathbb{E}[\rho_V(\Lambda)^2]=\vol(B)^2+\sum_{\alpha\in K^\times}D(\alpha)^{-t}\cdot\vol(B\cap\alpha^{-1} B),$$
where $D(\alpha)=[\OK:\alpha^{-1}\OK \cap\OK]$ is a lattice index measuring the denominator of $\alpha\in K^\times$ and where the scaling $\alpha^{-1} B$ occurs in $\KR$. In other words, $\alpha^{-1}$ multiplies the unit ball by $\sigma(\alpha^{-1})$ for the coordinate of $\KR$ corresponding to the field embedding $\sigma:K\to \IC$.\par
Since every $\alpha\in\mu(K)$ contributes exactly $\vol(B)$ to the sum, our task in achieving \eqref{eq:toshowmoment} reduces to estimating the  alleged error term:
\begin{equation}\label{eq:introoutline}
    \sum_{\alpha\in (K^\times\setminus\mu(K))/\mu(K)}D(\alpha)^{-t}\cdot\frac{\vol(B\cap\alpha^{-1} B)}{\vol(B)}.
\end{equation}
Suppose now that we are simply summing over units $\alpha\in\OK^\times$, so that we must estimate $\sum_{\alpha\in \OK^\times/\mu(K)}\frac{\vol(B\cap\alpha^{-1} B)}{\vol(B)}$. The body $B\cap\alpha^{-1} B$ intuitively should be comparable to the ellipsoid where for every embedding $\sigma:K\to \IC$ the length of the semi-axis in that direction is the radius of $B$ divided by $\max(1,\vert\sigma(\alpha)\vert)$. This naturally leads one to consider the notion of Weil height, a measure of the complexity of an algebraic number, and we set: 
$$H_\infty(\alpha) := \prod_{\sigma: K \rightarrow \mathbb{C}} \max\{1, |\sigma(\alpha) | \}.$$
This is closely related to the notion in the literature of (unnormalised) exponential Weil height of an algebraic number $\alpha\in K^\times$ given by the product over the set of places $M_K$ of $K$:
\begin{equation}
H_W(\alpha) := \prod_{v\in M_K} \max\{1, |\alpha |_v\}.
\end{equation}
It is easily seen that $H_W(\alpha)= H_\infty(\alpha)$ for $\alpha\in\OK$ and in general the two differ by a denominator. 
It is customary to work with \textbf{absolute} notions of height which do not depend on the particular subfield one is viewing an algebraic number in and to take logarithms. Writing $\deg(\alpha)=[\mathbb{Q}(\alpha):\mathbb{Q}]$ we arrive at the \textbf{Weil height}
$$h(\alpha)=\ln\left(H_W(\alpha)^{1/\deg(\alpha)}\right)$$
of an algebraic number. For non-integers, we write $h_\infty(\alpha)$ for $\ln\left(H_\infty(\alpha)^{1/\deg(\alpha)}\right)$. It follows from these definitions that we have the relation 
\begin{equation}
    h(\alpha)=h_\infty(\alpha)+ \tfrac{1}{\deg(K)}\cdot \ln(D(\alpha)), 
\end{equation}
where recall that we write $D(\alpha)=[\OK:\alpha^{-1}\OK \cap\OK]$ for the contribution from denominators at non-Archimedean places. 
\par

We intuitively expect to relate  $\frac{\vol(B\cap\alpha^{-1} S)}{\vol(B)}$ to $h_\infty(\alpha)$. Denoting henceforth by $\N(\alpha)$ the norm of the fractional ideal generated by $\alpha$, we indeed have the following result (see \cite[Lemma 35]{GSV23} for the proof of a more general statement, which draws on the arithmetic/geometric mean inequality): 
\begin{lemma}\label{lemma:genvolumeratio}
Let $\alpha\in K^\times$. We have the bound:
$$
\frac{\vol(B \cap \alpha^{-1}B)}{\vol(B)}\le\left(\frac{e^{2h_\infty(\alpha)}+e^{-2h_\infty(\alpha)}\cdot \N(\alpha)^{\frac{2}{d}}}{2}\right)^{-dt/2}$$
and moreover, under the assumption that $\N(\alpha)\geq 1$, we have for any $k\geq 2$:
$$ \frac{\vol(B \cap \alpha^{-1}B)}{\vol(B)}\le\N(\alpha)^{\frac{-t}{k}}\cdot \left(\frac{e^{2h_\infty(\alpha)\cdot (k-1)/k}+e^{-2h_\infty(\alpha)\cdot (k-1)/k}\cdot \N(\alpha)^{\frac{2}{d}}}{2}\right)^{-dt/2}.$$

\end{lemma}
In light of this result, our moment estimates will be related to height lower bounds for the number fields involved. Fortunately, this is a well-studied subject. For instance, Lehmer's famous problem asks for a \textbf{uniform} lower bound for $h(\alpha)\deg(\alpha)$ for all algebraic numbers which are not roots of unity (the latter are easily seen to have trivial height). At any rate, for a fixed number field, we may always find:  
\begin{equation}\label{eq:cdef}
    0<c(K)<\inf_{\alpha\in K^\times\setminus \mu(K)} h(\alpha).
\end{equation}
Indeed, for instance P. Voutier \cite{Voutier96}, building on work of Dobrowolski \cite{D1979}, showed that any
\begin{equation}\label{eq:cbound}
    c(K)< \frac{1}{4\deg(K)}\cdot \left(\frac{\ln\ln \deg(K)}{\ln  \deg(K)}\right)^3
\end{equation}
satisfies \eqref{eq:cdef}.
\subsection{Unit counts and unit contribution}
One of the difficulties in this setting is dealing with contributions from infinite unit groups. The next key technical ingredient is an upper bound for the number of units of bounded height: 
\begin{lemma}\label{le:unitcountSK}
    Let $S_K $ denote a fixed, finite subset of $\OK^\times $ that is closed under inversion. Let  $c_S(K)<\inf_{\alpha\in \OK^\times\setminus S_K}h_\infty(\alpha)$ be a positive height lower bound valid outside of $S_K$. 
    Consider the canonical log embedding: 
$L:K^\times\to \mathbb{R}^{r_1+r_2}$ defined by mapping
$$\alpha\mapsto (\ln \vert\sigma_1(\alpha)\vert,\ldots,2\ln \vert\sigma_{r_1+r_2}(\alpha)\vert),$$
as well as the function
\begin{align}
    h: \mathbb{R}^{r_1+r_2}&\to \mathbb{R}_{\geq 0} \\
    x&\mapsto \tfrac{1}{[K:\mathbb{Q}]}\cdot\sum_{j=1}^{r_1+r_2}\max(0,x_{j}).
\end{align}
    Then for any $\eta\in \mathbb{R}^{r_1+r_2}$ with $\sum_{j=1}^{r_1+r_2}\eta_j\geq 0$ and any $X\geq 0$ we have that
    $$\card \{\beta\in\mathcal{O}_K^\times~\vert~ h(\eta+L(\beta))\leq X\}\leq \# S_K\cdot \left(\frac{X+c_S(K)/2}{c_S(K)/2}\right)^{r_1+r_2-1}.$$
\end{lemma}
\begin{proof}
Note that the product formula implies that $L(\mathcal{O}_K^\times)$ is contained in the hyperplane $H:=\{x\in \mathbb{R}^{r_1+r_2}: \sum_{j=1}^{r_1+r_2}x_{j}=0\}$. Furthermore observe that $h$ satisfies the properties of a semi-norm on $H$. 
By assumption, for any $\beta\in \mathcal{O}_K^\times\setminus S_K$ we have that 
$$h(L(\beta))=h_\infty(\beta)>c_S(K).$$
Let now $P=\{\xi\in H: h(\xi)\leq c_S(K)/2\}$. We claim that for $\eta\in \mathbb{R}^{r_1+r_2}$ and $\beta_1,\beta_2\in \mathcal{O}_K^\times$: 
$$(L(\beta_1)+\eta+P)\cap (L(\beta_2)+\eta+P)=\emptyset \text{ unless }\beta_1\beta_2^{-1}\in S_K
$$

To prove the claim, let $y$ be in the intersection. Then by the triangle inequality 
$$h(L(\beta_1^{-1}\beta_2))=h(L(\beta_2)-L(\beta_1))\leq h(L(\beta_2)+\eta-y)+h(-L(\beta_1)-\eta+y).$$
Given that $L(\beta_2)+\eta-y\in P$ and $-L(\beta_1)-\eta+y\in -P$ we deduce that $h(L(\beta_1^{-1}\beta_2))\leq c_s(K)$ and therefore $\beta_1\beta_2^{-1}\in S_K$, proving the claim.

Moreover, any $\beta\in \mathcal{O}_K^\times$ satisfying $h(\eta+L(\beta))\leq X$ implies that $L(\beta)+\eta+P$ is contained in the set 
$$Q=\{\xi \in \mathbb{R}^{r_1+r_2}:  \sum_{j=1}^{r_1+r_2}\xi_{j}=Y_\eta, h(\xi)\leq X+c_S(K)/2 \},$$
where $Y_\eta=\sum_{j=1}^{r_1+r_2}\eta_j$. For any fixed $\eta$, we thus obtain by the claim that 
$$\card \left\{\beta\in\mathcal{O}_K^\times\vert h(\eta+L(\beta))\leq X\right\}\leq \# S_K\cdot\frac{\vol(Q)}{\vol(P)},$$
where the volumes are computed with respect to the natural measure identifying the hyperspaces $P$ and $Q$ are in with $\mathbb{R}^{r_1+r_2-1}$. \par 
For $\eta$ such that $Y_\eta=\sum_{j=1}^{r_1+r_2}\eta_j\geq 0 $, it is easy to see that the volume of $Q$ is bounded by the volume of $P_X= \{\xi\in H: h(\xi)\leq c_S(K)/2+X\}$. 
Thus we bound the desired unit count by 
$$\# S_K\cdot\frac{\vol(P_X)}{\vol(P)}=\# S_K\cdot\frac{\vol(\{\xi\in H: h(\xi)\leq c_S(K)/2+X\})}{\vol(\{\xi\in H: h(\xi)\leq c_S(K)/2\})}.$$
Since $H$ is an $\mathbb{R}$-vector space of dimension $r_1+r_2-1$ and $h$ a semi-norm on that vector space, the result follows. \par

\end{proof}
\begin{remark}
This is a slight generalization of \cite[Lemma 37]{GSV23} which establishes the result above for $S_K=\mu(K)$ the group of roots of unity in $K$. Allowing a slightly larger $S_K$ will allow to increase $c_S(K)$ and lead to better bounds. Note also that we must have $\mu(K)\subset S_K$ to have a positive height bound outside $S_K$.
\end{remark}
In summary, we will prove our results for the following setup, fixing notations: 
\begin{setup}\label{setupcyclo}
 Given a number field $K$, and a choice of finite subset $S_K$ closed under inversion and satisfying $\mu(K)\subseteq S_K\subset \OK^\times$, 
 there exists a choice of constants
 $$0<c(K)\leq c_o(K)\leq c_S(K)$$
 satisfying:
 \begin{eqnarray*}
      c(K)<\inf_{\alpha\in K^\times\setminus \mu(K)} h(\alpha)\\
      c_o(K)<\inf_{h_\infty(\alpha)> 0} h(\alpha)\\
      c_S(K)<\inf_{\alpha\in \OK^\times\setminus S_K}h(\alpha).
 \end{eqnarray*}
\end{setup}
It follows from our previous remarks that any fixed number field satisfies such a setup: 
\begin{remark}For any fixed number field $K$, the choice of $S_K=\mu(K)$, as well as $c_S(K)=c_o(K)=c(K)>0$ small enough as in  \eqref{eq:cdef} satisfies Setup \ref{setupcyclo}.
    
\end{remark}
However, for many number fields such as cyclotomic fields, making these distinctions will allow us to prove better bounds. Moreover, one can in many instances choose the constants uniformly for a family of number fields $K$. \par
We now turn to our main results in estimating the second moment. We note that there are two points of view that one may be interested in: first, we tailor our results to the perspective of varying the number fields involved, and aim to achieve uniform control of the moments throughout the family for $\OK$-rank as small as possible. Second, for fixed number field and fixed (or increasing) rank the results can be simplified-this will be discussed in Section \ref{sec:example}.

\begin{proposition}\label{prop:mintersections}
    Let a number field $K$ and let bounds as in setup \ref{setupcyclo} together with its notations be given. For any $k\geq 2$ there exist positive constants $C,\varepsilon_1>0$ uniformly bounded in $t$ such that for all $\alpha\in K^\times\setminus \mu(K)$ with $N(\alpha)\geq 1$ the following holds:  write
     $$t_0=\frac{2r_K}{d}\cdot\frac{ \ln(1+\frac{2c_o(K)}{c_S(K)}+\tfrac{1}{k^3})}{\ln(\cosh(2\cdot c_o(K)\cdot(1-\frac{1}{k}))},$$
     where $r_K$ is the rank of the unit group of $K$.
Then we have for any $t> t_0$ and any $d\geq 1$ that
\begin{align}
 \sum_{\substack{\beta\in\mathcal{O}_K^\times\\ \alpha\beta\notin \mu(K)}}\frac{\vol(B\cap(\alpha\beta)^{-1}B)}{\vol(B)} 
 \leq 
 C\cdot\card S_K\cdot \N(\alpha)^\frac{-t}{k}\cdot D(\alpha)^{\frac{t}{2}}\cdot  e^{-\varepsilon_1\cdot d\cdot (t-t_0)}.
\end{align}
Moreover, the constants can be made explicit. We may for instance take 
$$\varepsilon_1=\tfrac{5\cdot c_o(K)^2(k-1)^2}{6\cdot k^2}$$ and $C=1+k^{3}+\tfrac{1}{1-e^{-5\cdot c_o(K)^2\cdot d(t-t_0)(k-1)^2/(3 k^5)}}$.
 
\end{proposition}
\begin{proof}
The proof consists of using Lemma \ref{le:unitcountSK} to bound the number of units in a given height range, which grows exponentially in $d$, and compare it with the forced decay due to height (and denominator) from our geometric Lemma \ref{lemma:genvolumeratio}, the latter decay being exponential in $d\cdot t$. This yields a minimal rank $t_0$ for which when $t>t_0$ our sum over the unit group can be bounded by a geometric sum. We deal with different height ranges separately in order to make the minimal rank as small as possible. \par 
More precisely, by Lemma \ref{lemma:genvolumeratio} our task is reduced to bounding for suitably chosen $k\geq 2$: 
    \begin{equation}\label{eq:projtobound}
        \sum_{\substack{\beta\in\mathcal{O}_K^\times \\ \alpha\beta \notin \mu(K)}}\N(\alpha)^\frac{-t}{k}\cdot \cosh\left( h_\infty(\alpha\beta)\left( 2(1-\tfrac{1}{k})\right)\right)^{-dt/2}.
    \end{equation}
     Since $\cosh(x)$ is increasing for $x>0$ it suffices to bound, for any $A \in \mathbb{Z}_{>0}$, 
    the sum
$$\Sigma_{k}^\infty=
\sum_{m=A}^{\infty}\card\left\{\beta\in\mathcal{O}_K^\times\mid h_\infty(\alpha\beta)\in \Big[\tfrac{mc_S(K)}{A}, \tfrac{(2m+1)c_S(K)}{2A}\Big[~\right\} \cdot 
\cosh\left(\frac{2mc_S(K)}{Ak/(k-1)}\right)^{\frac{-dt}{2}}$$
    together with the contribution of the remaining units satisfying $h_\infty(\alpha\beta)< c_S(K)$:
   
	    \begin{equation}
	      \Sigma_{k}^{c_S} = 
\sum_{\substack{  \beta\in\mathcal{O}_K^\times \\ h_\infty(\alpha\beta)< c_S(K)~ }} 
\cosh\left(  h_{\infty}(\alpha \beta)  \cdot  2 (1-\tfrac{1}{k})\right)^{-dt/2}.
	    \end{equation}

	    Let us examine the term $\Sigma_{k}^{c_S}$ first. For points of very small height we can show (using the triangle inequality for heights as in Lemma \ref{le:unitcountSK}) that 
\begin{equation}\label{eq:polycount}\card \{\beta\in\mathcal{O}_K^\times\vert h_\infty(\alpha\beta)\leq c_S(K)/2\}\leq \card S_K.
\end{equation}
   
 Moreover, by assumption, since $\alpha\beta\notin\mu(K)$ we have the height bound
  \begin{equation}\label{eq:genbound}
      h(\alpha\beta)=h_\infty(\alpha\beta)+ \tfrac{1}{d}\cdot \ln D(\alpha)\geq c(K)>0.
  \end{equation}
The terms of very small height therefore satisfy: 
\begin{align}
    &D(\alpha)^{-\frac{t}{2}}\cdot\sum_{\substack{  \beta\in\mathcal{O}_K^\times \\ h_\infty(\alpha\beta)< c_S(K)/2~ }}\cosh\left(  h_{\infty}(\alpha \beta)  \cdot  2 (1-\tfrac{1}{k})\right)^{-dt/2}\\
    &\leq \card S_K\cdot\exp\left(-dt\cdot\min\left(\tfrac{c(K)}{2}, \tfrac{1}{2}\cdot\ln(\cosh(2c(K)(\tfrac{k-1}{k}))\right)\right)
\end{align}
and are controlled for any rank $t$ when $\card S_K$ grows polynomially in $d$.
We now turn to the contribution of terms with $h_\infty(\alpha\beta)\in [c_S(K)/2,c_S(K)]$ which we partition into $A$ intervals $I_i=[\tfrac{c_S(K)(1+i/A)}{2},\tfrac{c_S(K)(1+(i+1)/A)}{2}]$ for $i=0,\ldots, A-1$.
By Lemma \ref{le:unitcountSK} we may bound
\begin{equation}\label{eq:ln3count}
    \card \{\beta\in\mathcal{O}_K^\times\vert h_\infty(\alpha\beta)\in I_i\}\leq \card S_K\cdot \left(2+\tfrac{i+1}{A}\right)^{r_K}.
\end{equation}
Recall that in our setup for  $h_\infty(\alpha\beta)\in I_i$ we have a combined bound on the height at Archimedean places and the denominators:
        $$h_\infty(\alpha\beta)+ \tfrac{1}{d}\cdot \ln D(\alpha)\geq c_o(K).$$
Therefore, we may estimate the contribution on each interval by 
\begin{align}
    &D(\alpha)^{-\frac{t}{2}}\cdot\sum_{\substack{  \beta\in\mathcal{O}_K^\times \\ h_\infty(\alpha\beta)\in I_i~ }}\cosh\left(  h_{\infty}(\alpha \beta)  \cdot  2 (1-\tfrac{1}{k})\right)^{-dt/2}\\
    &\leq \card S_K\cdot \left(2+\tfrac{i+1}{A}\right)^{r_K}\cdot \exp\left(-dt\cdot\left(D_{sav}+ \tfrac{1}{2}\cdot\ln(\cosh(\tfrac{c_S(K)(A+i)(k-1)}{Ak}))\right)\right),
\end{align}
where we wrote for savings from denominator terms  $$D_{sav}=\max\left(0,\tfrac{c_o(K)-c_S(K)(\frac{A+1+i}{2A})}{2}\right).$$
A small calculation shows that largest contribution comes from the interval with $c_o(K)\in I_i$, as it is the smallest $i$ with $D_{sav}=0$. We therefore arrive at the minimal $\OK$-rank condition of:
$$t_0\geq\frac{2r_K}{d}\cdot \frac{ \ln(1+\frac{2c_o(K)}{c_S(K)}+\tfrac{1}{A})}{\ln(\cosh(2\cdot c_o(K)\cdot(1-\frac{1}{k}))}.$$

    It now follows that for $t>t_0$ the sum $D(\alpha)^{-\frac{t}{2}}\cdot \Sigma_{k}^{c_S}$ is bounded as claimed in the proposition. 
    \par 
    Finally, we turn to bounding the tail $\Sigma_{k}^{\infty}$. Lemma \ref{le:unitcountSK} with height bound $X=\tfrac{(2m+1)\cdot c_S(K)}{2A}$ yields:    
	    \begin{align}
		    \Sigma_{k}^\infty
		    &\leq \card S_K\cdot \sum_{m=A}^{\infty}\left(\tfrac{2m+A+1}{A}\right)^{r_K}\cdot \cosh(\tfrac{2m\cdot (k-1)}{A\cdot k}\cdot c_S(K))^{-dt/2}\\
		    &= \card S_K\cdot \sum_{n=1}^{\infty}\left(\tfrac{2n-1}{A}+3\right)^{r_K}\cdot \cosh(\tfrac{A+n-1}{A}\cdot 2c_S(K)\cdot (1-\tfrac{1}{k}))^{-dt/2}\\
				       &\leq \card S_K\cdot \sum_{n=1}^{\infty} \cosh(\tfrac{A+n-1}{A}\cdot 2c_S(K)\cdot (1-\tfrac{1}{k}))^{\frac{-d(t-t_0)}{2}}
				       \label{eq:geom_series}
	    \end{align}
    where we chose
    $$t_0\geq\frac{2r_K}{d}\cdot \sup_{n\in\mathbb{N}_{\geq 1}}\frac{ \ln(\frac{2n-1}{A}+3)}{\ln \left(\cosh(\frac{A+n-1}{A}\cdot 2\cdot c_S(K) \cdot (1-\frac{1}{k}) \right)}$$
    for suitable $k$. The logarithm ratio decays as $n$ increases and therefore it suffices to take 
     $$t_0\geq\frac{2r_K}{d}\cdot \frac{ \ln(\frac{1}{A}+3)}{\ln(\cosh(2\cdot c_S(K)\cdot(1-\frac{1}{k}))}.$$

     The sum in  \eqref{eq:geom_series} can now be bounded by a geometric series, for instance we may bound $\ln(\cosh(x))\geq \frac{5}{12}\min(x,x^2)$. Evaluating the geometric series and some bookkeeping yields the explicit constants above (we chose to present the results for the choice of $A=k^3$). 
   
\end{proof}
\subsection{Summing contributions over ideals}
Having obtained in proposition \ref{prop:mintersections} bounds for the second moment error terms for cosets modulo algebraic units, we now upgrade these sums to all contributions over $K^\times$ by summing over (principal) ideals. We utilize in particular the summation formula proved more generally in \cite[Proposition 42]{GSV23}: 
\begin{proposition}\label{prop:reformulation}
Let $\alpha\in K^\times$. For any function 
$f:K^{\times}\to \mathbb{R}$ and any $T\in \mathbb{R}_{> 1}$, the sum 
$$\sum_{\alpha\in K^\times}D(\alpha)^{-T}\cdot \vol(B\cap \alpha^{-1}B)\cdot f(\alpha)$$
equals: 
$$\zeta_K(T)^{-1}\cdot \sum_{\substack{\mathcal{I}\subset \calO_K\\ \mathcal{I} \text{integral ideal}}}\N(\mathcal{I})^{-T}\sum_{\alpha\in \mathcal{I}^{-1}\setminus\{0\}} \vol(B\cap \alpha^{-1}B)\cdot f(\alpha).$$
\end{proposition}
Here, $\zeta_K(s)$ denotes the Dedekind zeta function of $K$. With this in hand, we obtain our main moment error term estimate:

\begin{proposition}\label{prop:projsumoverideals}
Let a number field $K$ and let bounds as in setup \ref{setupcyclo} together with its notations be given. For any $k\geq 3$ there exist positive constants $C,\varepsilon_1>0$ uniformly bounded in $t$ such that the following holds:  write
$$t_0=\max\left\{k,6,\frac{2r_K}{d}\cdot\frac{ \ln(1+\frac{2c_o(K)}{c_S(K)}+\tfrac{1}{k^3})}{\ln(\cosh(2\cdot c_o(K)\cdot(1-\frac{1}{k}))}\right\},$$
where $r_K$ is the rank of the unit group. For any $t> t_0$ we then have:
$$\sum_{\alpha\in K^\times\setminus\mu(K)} D(\alpha)^{-t}\frac{\vol(B \cap \alpha^{-1} B)}{\vol(B)}\leq C\card S_K\cdot\frac{\zeta_K(t(\frac{1}{2}-\frac{1}{k}))\cdot \zeta_K(\frac{t}{k})}{\zeta_K(\frac{t}{2})}\cdot e^{-\varepsilon\cdot d\cdot(t-t_0)}. $$
 We may moreover take 
$$\varepsilon=\tfrac{5\cdot c_o(K)^2(k-1)^2}{6\cdot k^2}$$ and $C=2\left(1+k^{3}+\tfrac{1}{1-e^{-5\cdot c_o(K)^2\cdot d(t-t_0)(k-1)^2/(3 k^5)}}\right)$.
\end{proposition}

\begin{proof}

   It is easy to see that the quantity $D(\alpha)^{-t}\vol(B \cap \alpha^{-1} B)$ is invariant under inversion $\alpha\mapsto\alpha^{-1}$, so that we may just bound by twice the sum under the additional assumption $\N(\alpha)\geq 1$. Proposition \ref{prop:mintersections} (constants $C,\varepsilon_1$ as in \ref{prop:mintersections}), reduces our task to estimating the sum
  
  $$ \sum_{\substack{\alpha\in K^\times/\calO_K^\times\\ \N(\alpha)\geq 1}} C\cdot D(\alpha)^{-t}\N(\alpha)^\frac{-t}{k}\cdot D(\alpha)^{\frac{t}{2}}\cdot e^{-\varepsilon_1\cdot d\cdot(t-t_0)}.$$
  By Proposition \ref{prop:reformulation} with $T=\tfrac{t}{2}$, it therefore suffices to bound 

$$\zeta_K(\tfrac{t}{2})^{-1}\cdot \sum_{\substack{\mathcal{I}\subset \calO_K\\ \mathcal{I} \text{ integral ideal}}}\N(\mathcal{I})^{-\frac{t}{2}}\sum_{\substack{\alpha\in (\mathcal{I}^{-1}\setminus\{0\})/\calO_K^\times\\\\ \N(\alpha)\geq 1}}\N(\alpha)^\frac{-t}{k}\cdot e^{-\varepsilon_1\cdot d\cdot(t-t_0)} .$$
Now observe that the map $\alpha\mapsto(\alpha)\cdot \mathcal{I}$ gives a bijection between $(\mathcal{I}^{-1}\setminus\{0\})/\calO_K^\times$ and integral ideals $\mathcal{J}\subset \calO_K$ in the ideal class of $\mathcal{I}$. We may therefore bound this expression by: 

$$\zeta_K(\tfrac{t}{2})^{-1}\cdot\sum_{\substack{\mathcal{I} \subset \calO_K\\ \mathcal{I} \text{ integral ideal}}}\N(\mathcal{I})^{-\frac{t}{2}}\sum_{\substack{\mathcal{J}\subset \calO_K\\ \N(\mathcal{I})\leq \N(\mathcal{J})}}\N(\mathcal{J}\mathcal{I}^{-1})^{-\frac{t}{k}}\cdot e^{-\varepsilon_1\cdot d\cdot(t-t_0)}$$
and therefore as claimed by 
$$ \frac{\zeta_K(t(\frac{1}{2}-\frac{1}{k}))\cdot \zeta_K(\frac{t}{k})}{\zeta_K(\frac{t}{2})}\cdot e^{-\varepsilon_1\cdot d\cdot(t-t_0)}.$$
The result follows (with some bookkeeping of the claimed constants) by noting that Dedekind zeta functions converge for real values $s>1$.  
  
\end{proof}
We may therefore put everything together and obtain the second moment result in its abstract form for an arbitrary number field: 
\begin{theorem}\label{thm:mainsecondmoment}
  Let a number field $K$ and let bounds as in setup \ref{setupcyclo} together with its notations be given. For any $k\geq 3$ there exist positive constants $C,\varepsilon_1>0$ uniformly bounded in $t$ such that the following holds:  write
$$t_0=\max\left\{k,6,\frac{2r_K}{d}\cdot\frac{ \ln(1+\frac{2c_o(K)}{c_S(K)}+\tfrac{1}{k^3})}{\ln(\cosh(2\cdot c_o(K)\cdot(1-\frac{1}{k}))}\right\},$$
where $r_K$ is the rank of the unit group. For any $t> t_0$ we then have that the second moment $\mathbb{E}[\rho_V(\Lambda)^2]$ of the number of nonzero lattice points in a fixed origin-centered ball of volume $V$ in $K_\mathbb{R}^t$ for a Haar-random module lattice of rank $t$ over $\calO_K$ satisfies: 
\begin{align}
       V^2+\omega_K\cdot V&\leq \mathbb{E}[\rho_V(\Lambda)^2]\\
       &\leq   V^2+\omega_K\cdot V+\omega_K\cdot S_K \cdot Z(K,t,k)\cdot e^{-\varepsilon\cdot d\cdot (t-t_0)}\cdot V,\\
    \end{align}
    where $0\leq Z(K,t,k)\leq\zeta_K\big(t(\frac{1}{2}-\frac{1}{k})\big)\cdot \zeta_K(\frac{t}{k})\cdot\zeta_K(\frac{t}{2})^{-1}$. \par
     We may moreover take 
$$\varepsilon_1=\tfrac{5\cdot c_o(K)^2(k-1)^2}{6\cdot k^2}$$ and $C=2\left(1+k^{3}+\tfrac{1}{1-e^{-5\cdot c_o(K)^2\cdot d(t-t_0)(k-1)^2/(3 k^5)}}\right)$.
   
\end{theorem}
We conclude this section with a tail bound of sorts that bounds the contribution to the error term of algebraic numbers in $K$ above some height threshold. This will be useful when considering a fixed number field for which estimates on the contribution of small height algebraic numbers can be computed explicitly.
\begin{proposition}\label{prop:tailbounds}
For a number field $K$ of degree $d$ consider the setup \ref{setupcyclo} and its notations. Let $h_0\geq c_S(K)>0$ be some threshold Weil height. Assume for simplicity $h_0\leq 2$ and write 
$$t_0=\max\left\{4,\frac{2r_K}{d}\cdot\frac{ \ln(1+\frac{12h_0+1}{6c_S(K)})}{\ln(\cosh(\tfrac{3h_0}{2}))}\right\}$$
For any $t> t_0$ we then have:
$$\sum_{h(\alpha)>h_0} D(\alpha)^{-t}\vol(B \cap \alpha^{-1} B)\leq C\cdot\card S_K\cdot\frac{\zeta_K(\frac{t}{4})^2}{\zeta_K(\frac{t}{2})}\cdot e^{-\varepsilon\cdot d\cdot(t-t_0)}\cdot \vol(B). $$
 We may moreover take 
$$\varepsilon=\tfrac{1}{2}\ln(\cosh(\tfrac{3h_0}{2}))\text{ and }C=1+\tfrac{2}{1-e^{- 5d(t-t_0)/192}}.$$
\end{proposition}
\begin{proof}
The proof proceeds along the lines of bounding sums over unit cosets as in the proof of Proposition \ref{prop:mintersections}, avoiding the trickier small heights and choosing $k=4$ for simplicity. By assumption, we have that 
$$h(\alpha)=h_\infty(\alpha)+\tfrac{1}{d}\ln D(\alpha)\geq h_0.$$
Using our methods, the worst bound for the sum over units 
$$\Sigma_0:=D(\alpha)^{-t/2}\sum_{\substack{\beta\in\mathcal{O}_K^\times\\ \alpha\beta\notin \mu(K)}}\frac{\vol(B\cap(\alpha\beta)^{-1}B)}{\vol(B)}$$ 
is, under the assumption $h_0\leq 2$, for trivial denominator $D(\alpha)$. The bound 
$$\Sigma_0\leq \sum_{i=0}^\infty  \card \{\beta\in\mathcal{O}_K^\times\vert h_\infty(\alpha\beta)\in [h_0+\tfrac{i}{A},h_0+\tfrac{i+1}{A})\}\cdot\cosh(\tfrac{3}{2}\cdot (h_0+\tfrac{i}{A}))^{-dt/2}$$
therefore holds regardless and for any $A>0$. The unit counts are estimated using Lemma \ref{le:unitcountSK} and the resulting geometric sum evaluated as in Proposition \ref{prop:mintersections} for the choice of $A=12$. Finally, we sum these bounds over principal ideals as in Proposition \ref{prop:projsumoverideals}, yielding the result. 
    
\end{proof}
\subsection{Second moment for cyclotomic fields}
We now turn to establishing the results for cyclotomic fields, when good height lower bounds are available for setup \ref{setupcyclo}. 
Cyclotomic fields have complex multiplication (CM), and for all such fields, since complex conjugation $x\mapsto \overline{x}$ is in the center of the Galois group, we have the following property: 
\begin{lemma}\label{le:CMconjugates}
    For $\alpha$ in a CM field with conjugates $\alpha_j$, the conjugates of $\vert \alpha\vert^2=\alpha\overline{\alpha}$ are the $\vert \alpha_j\vert^2$. Moreover, if $\vert\alpha_j\vert^2 =q\in\mathbb{Q}$ for some $j$, the same holds for any conjugate.  
\end{lemma}
In particular, this will imply that height lower bounds for totally positive or totally real integers also apply to algebraic integers in a CM field. To that end, we have the foundational result due to A. Schinzel \cite{Schinzel1973}: 
\begin{theorem}\label{thm:cycloheightbound}
    Assume that an algebraic number $\alpha$ of infinite multiplicative order is contained in a totally real field. Then, denoting by $\varphi=\frac{1+\sqrt{5}}{2}$ the golden ratio, we have
     $$h(\alpha)\geq \frac{1}{2}\ln{\varphi}\approx 0.2406\ldots.$$
     Moreover, the same is true for $\alpha$ in a CM field provided one (and equivalently, all) of its Archimedean embeddings satisfy $\vert \alpha\vert\neq 1$. 
\end{theorem}
Going beyond the work of Schinzel, C.J. Smyth pioneered the study on the lowest possible heights (or, equivalently, Mahler measures) for totally positive algebraic integers, with some follow-up work by V. Flammang, among others. They show: 
\begin{theorem}[\cite{Smyth1981,FLAMMANG2015211}]\label{thm:heightexceptions}
    Suppose $\alpha$ is a nonzero totally positive algebraic integer whose minimal polynomial does not divide any of the following:
{\small\begin{align}
  &x- 1  \\
  &x^2-3x+1\\
  &x^4- 7x^3 + 13x^2-7x + 1\\
  &x^6 -11x^5 +41x^44-63x^33 +41x^2 -11x+1\\
  &x^8 -15x^7 + 83x^6-220x^5 + 303x^4 -220x^3 + 83x^2 -15x + 1 \\
  &x^8 -15x^7 + 84x^6-225x^5 + 311x^4 -225x^3 + 84x^2 -15x + 1 \\
  &x^{16}-31x^{15} + 413x^{14}-3141x^{13}+15261x^{12}-50187x^{11}+115410x^{10}-189036x^{9}\\
  &+222621x^{8}-189036x^{7}+115410x^{6}-50187x^{5} + 15261x^{4}-3141x^{3} + 413x^{2} -31x + 1.
\end{align}  }
We then have the height bound $h(\alpha)>0.543526$. 
\end{theorem}
It follows that for totally real and CM algebraic integers (including cyclotomic ones), half the bound above, namely 
$$h(\alpha)>0.271763$$
holds up to a fixed, finite list of exceptions. What is more, the result above is not far from optimal in the sense that Smyth showed that the Weil heights of totally real integers are everywhere dense in $(\ell,\infty)$ for a certain $\ell=0.27328\ldots$, see \cite[Theorem 1]{Smyth_1980}. \par

Removing any integrality assumption weakens the bounds, but in any abelian extensions Amoroso and Dvornicich \cite{AMOROSO2000260} showed:
\begin{theorem}\label{thm:cycloheightboundnumbers}
    Assume that an algebraic number $\alpha$ of infinite multiplicative order is contained in an abelian extension of $\mathbb{Q}$. Then 
     $$h(\alpha)\geq \frac{\ln{5}}{12}\approx 0.1341\ldots.$$
\end{theorem}
\begin{remark}
     For cyclotomic fields, this bound was improved in \cite{ISHAK20101408} to $h(\alpha)\geq 0.155\ldots$ and stronger bounds can be proved if the conductor avoids small primes. 
\end{remark}

We summarize the most useful results in the literature for us in the case of cyclotomic fields: 
\begin{proposition}\label{prop:cycloheightbounds}
    Any cyclotomic field $K$ satisfies setup \ref{setupcyclo} with the constants 
    $$c(K)=0.155 \text{ and }c_o(K)=0.2406 \text{ and }c_S(K)=0.271763$$
    \textbf{uniformly} in $K$, with $S_K$ being the set of units $u\in \OK^\times$ such that totally positive $u\overline{u}$ has minimal polynomial listed among the exceptions in Theorem \ref{thm:heightexceptions}. In particular, we have $\card S_K\leq 17\cdot \card\mu(K)$. 
\end{proposition}
\begin{proof}
    Everything follows from the above discussion except the bound on $\card S_K$. Theorem \ref{thm:heightexceptions} gives an exhaustive list of totally positive integers not satisfying the height bound. Of the polynomials listed, those of degree $4,16$ and the first of degree $8$ do not generate abelian extensions. Setting $\alpha_m=(2\cos(2\pi/m))^2$, the remaining polynomials of degree $2,6,8$ are respectively the minimal polynomial of $\alpha_5$ and the product of minimal polynomials of $\alpha_7, \alpha_7^{-1}$ and $\alpha_{60}, \alpha_{60}^{-1}$. This gives at most $17$ totally positive integers to avoid. \par
    To count the exceptions in $K$, let $\alpha,\beta\in \OK^\times$ such that $\vert\alpha\vert^2 =\vert\beta\vert^2 $. Then $\alpha\beta^{-1}$ is an algebraic integer all of whose conjugates lie on the unit circle by Lemma \ref{le:CMconjugates}. Therefore $\alpha\beta^{-1}\in \mu(K)$ by Kronecker's first theorem. 
\end{proof}
We therefore arrive at our main result for cyclotomic fields: 
\begin{theorem}\label{thm:cyclotomicnew}
    Let $K$ denote a cyclotomic field containing $\omega_K$ roots of unity. The second moment $\mathbb{E}[\rho_V(\Lambda)^2]$ of the number of nonzero lattice points in a fixed origin-centered ball of volume $V$ in $K_\mathbb{R}^t$ for a Haar-random module lattice of rank $t\geq 11$ over $\calO_K$ satisfies: 
\begin{align}
       V^2+\omega_K\cdot V&\leq \mathbb{E}[\rho_V(\Lambda)^2]\\
       &\leq   V^2+\omega_K\cdot V+\omega_K^2\cdot C \cdot Z(K,t,k)\cdot e^{-\varepsilon\cdot d\cdot (t-10.99)}\cdot V,\\
    \end{align}
    with constants given by 
    $$\varepsilon=\tfrac{1}{26}\text{ and }C=68\cdot \left(1332+\tfrac{1}{1-e^{-\cdot d(t-10.99)/(16693)}}\right)$$
    and where $0\leq Z(K,t)\leq\zeta_K\big(\frac{9t}{22})\big)\cdot \zeta_K(\frac{t}{10.99})\cdot\zeta_K(\frac{t}{2})^{-1}$.
\end{theorem}
\begin{proof}
    We apply the bounds of Proposition \ref{prop:cycloheightbounds} to Theorem \ref{thm:mainsecondmoment} and choose $k<t$ large enough to achieve $t_0=11$. 
\end{proof}
\begin{remark}
    Theorem \ref{thm:cyclotomicnew} is geared towards obtaining convergence to the second moment of a Poisson distribution of mean $V/\omega_K$ for large dimension without having to take the $\OK$-rank $t$ large. Should one instead be interested in a fixed cyclotomic field and increasing solely the rank, one can typically arrive at slightly better height bounds than proposition \ref{prop:cycloheightbounds}, see Section \ref{sec:example}. One can also obtain a smaller constant $C$ if one is willing to compromise on the minimal rank required. 
\end{remark}
\section{Shortest vector bounds}\label{sec:SVbounds}

Having relatively tight bounds for the second moment in hand, the following result will allow us to estimate the shortest vector. The result is most naturally stated in terms of the volume minimum $\mathcal{V}_1(\Lambda)=\lambda_1(\Lambda)^n\cdot \vol(\mathbb{B}_n(1))$ of a random lattice, where $\vol(\mathbb{B}_n(1))$ is the volume of an $n$-dimensional unit ball:  
\begin{proposition}\label{prop:probabounds}
Let $K$ be a number field of degree $d$ with $\omega_K=\# \mu(K)$ roots of unity and let $\Lambda \in \ML_{t}(K)$ be a Haar-random lattice in dimension $n=dt$. Then for any $\varepsilon(n)<1$ the first volume minimum $\mathcal{V}_1(\Lambda)$ of $\Lambda$ satisfies 
$$\omega_K\cdot \varepsilon(n)\leq \mathcal{V}_1(\Lambda)\leq \omega_K\cdot\varepsilon(n)^{-1}$$
with probability greater than $1-\varepsilon(n)\cdot(2+\omega_K\cdot\Err(K,t))$ provided that the error term $\Err(K,t)$ is such that  
$$\mathbb{E}[\# ( \Lambda \cap B \setminus \{0\} )^2]\leq V^2+\omega_K\cdot V\cdot (1+\Err(K,t))$$
for any ball $B$ of volume $V$. 
\end{proposition}
\begin{proof} 

 Let $\Lambda \in \ML_{t}(K)$ be a Haar-random lattice and write $\rho_V(\Lambda) = \card ( \Lambda \cap B \setminus \{0\} )$.
 Note that because of (generalizations of) Siegel's mean value theorem, we have that $\mathbb{E}(\rho_V(\Lambda))  = V$. Moreover we know that $\rho_V(\Lambda)\in \omega_K\cdot \mathbb{Z}_{\geq 0}$ because the roots of unity in $K$ act on lattice vectors preserving lengths. 

 Let $m \geq 1$ be arbitrary. For $x \in \mathbb{Z}_{\geq 0}$, observe that the following inequality holds concerning the indicator function of $x=0$:
 \begin{equation}
1- x\leq \mathbf{1}( x=0 ) \leq  {m}^{-2}(x-m)^{2}.
\label{eq:sandwich}
 \end{equation}
 We now set $x= \rho_V(\Lambda)/k$ and evaluate expected values. 
 Using the left side of \ref{eq:sandwich}
 we get that $1-\tfrac{V}{\omega_K}\leq \mathbb{P}(x=0)$. Setting $V=\omega_K\cdot \varepsilon(n)$ yields that with a probability of at least $1-\varepsilon(n)$, the random variable $\rho_V(\Lambda) = 0$ and therefore the first volume minimum satisfies 
 \begin{equation}
 \omega_K\cdot \varepsilon(n)\leq \mathcal{V}_1(\Lambda).
 \end{equation}
 On the other hand, we have by our assumption on the second moment of $\rho_V(\Lambda)$: 
\begin{equation}
	 \mathbb{P}(\rho_V(\Lambda) = 0)   \leq  \frac{1}{m^{2}}\left(\frac{V^{2}}{\omega_K^{2}} + \frac{V}{\omega_K}\cdot \left( 1 + \omega_K\cdot \Err(K,t)\right) \right) - \frac{2\cdot V}{m\cdot \omega_K} +1 
 \end{equation}
Setting $V=\omega_K\cdot m$ therefore yields that 
\begin{equation}
    \mathbb{P}(\rho_V(\Lambda) \neq 0)\geq 1-\tfrac{1}{m}\cdot \left(1+\omega_K\cdot \Err(K,t)\right).
\end{equation}
For the choice of $m=\varepsilon(n)^{-1}$ it follows in other words that $$\mathcal{V}_1(\Lambda)\leq \omega_K\cdot\varepsilon(n)^{-1}$$
with probability greater than $1-\varepsilon(n)\cdot \left(1+\omega_K\cdot \Err(K,t)\right)$.
The result follows. 
 \end{proof}
It therefore follows that for \emph{any} fixed number field via the bound \eqref{eq:cbound}, we can control the second moment as the rank and therefore blocksize varies via Theorem \ref{thm:mainsecondmoment} and obtain good probabilistic bounds on the shortest vector. Similarly, whenever we have uniform height lower bounds, the same conclusion applies when varying the number field. We record here the corresponding result for cyclotomic fields: 
\begin{theorem}
	\label{th:as_module_bound}
Let $t \geq 11$ be fixed.
Let $K = \mathbb{Q}(\zeta_k)$ be a cyclotomic number field and let $n = t \cdot \deg K = t \varphi(k)$. Let $\omega_K=\card\mu(K)$ which is $k$ or $2k$ depending on whether $k$ is even or odd. Consider a Haar-random 
unit covolume module lattice $\Lambda \subseteq K^{t} \otimes \mathbb{R}$ in the moduli space of rank-$t$ module lattices over $K$.
Then, as $k \rightarrow \infty$, we have with probability $1- o(1)$ that
\begin{equation}
 1- \frac{\log \log \omega_K}{n} \leq \omega_K^{-\frac{1}{n}}\cdot\frac{\lambda_1(\Lambda)}{\gamma(n)} \leq  1 + \frac{\log \log \omega_K}{n}
 \label{eq:bound}
\end{equation}
where $\gamma(n)$ is the radius of a ball of unit volume in dimension $n$.
\end{theorem}
\begin{proof}
    The result follows from Theorem \ref{thm:cyclotomicnew} and Proposition \ref{prop:probabounds} with $\varepsilon(n)=\tfrac{1}{\ln(\omega_K)}$.
\end{proof}
We end by recording the following remark, which follows from examining the proof of Proposition \ref{prop:probabounds}: 
\begin{remark}\label{rem:unconditionallower}
    The lower bounds on the shortest vector in Proposition \ref{prop:probabounds} and Theorem \ref{th:as_module_bound} are valid unconditionally, or rather, as soon as $t\geq 2$.
\end{remark}
 Indeed only the exact result $\mathbb{E}[\rho_V(\Lambda)]=V$ result is used for that bound. 
\section{Remarks on fixed number fields}\label{sec:example}
Should the reader be interested in a fixed cyclotomic or other number field $K$ and varying or fixed rank, our results can be significantly sharpened by explicitly computing the contributions of low height elements in $K$. The following example can be reproduced for a variety of number fields with sufficiently small degree to make the computations feasible: 
\begin{example}\label{ex:fixedK}
    Consider the cyclotomic field $K=\Q(\zeta_{16})$ of degree $8$ and conductor $16$. We take the rank to be $t=32$ leading to random lattices in dimension $256$. 
    In a matter of seconds on a personal computer, using Lemma \ref{lemma:genvolumeratio} and a SageMath \cite{sagemath} implementation of algorithms for finding points of bounded height as layed out in \cite{ComputingKrummDoyle}, we find that the contribution to the error term of the points 
    $$\{\alpha \in \Q(\zeta_{16})^\times: h(\alpha)\leq 0.6\}$$
    is bounded by $1.195\cdot 10^{-11}$. The contribution of the remaining terms can be estimated by the tail estimate in Proposition \ref{prop:tailbounds} with $h_0=0.6$, $c_S(K)=0.271763$ and even $\card S_K=16$ since one can check none of the exceptional integers listed in Theorem \ref{thm:heightexceptions} occurs in $K$. The Dedekind zeta values can also be computed with reasonable precision in SageMath and are bounded by $\zeta_K(8)^2<1.01$ and we find that the tail contribution is of the order of $10^{-14}$. We obtain therefore that the space of random lattices of rank $32$ over $\mathbb{Z}(\zeta_{16})$ satisfies for any ball in dimension $256$ of volume $V$:
    $$\mathbb{E}[\# ( \Lambda \cap B \setminus \{0\} )^2]\leq V^2+\omega_K\cdot V\cdot (1+\eta)$$
    with $\eta\leq 1.2\cdot 10^{-11}$. \par
    Therefore, we obtain by Proposition \ref{prop:probabounds} shortest vector bounds essentially unaffected by the error term $\eta$. For instance if we make the choice of  $\varepsilon=\ln(256)^{-1}$, with probability greater than $1-\varepsilon(2+16 \eta)\geq 0.639$, the shortest vector of such random module lattices must satisfy 
    \begin{equation}
  1- \frac{\log \log 256}{256} \leq 16^{-\frac{1}{256}}\cdot\frac{\lambda_1(\Lambda)}{\gamma(256)} \leq  1 + \frac{\log \log 256}{256},
\end{equation}
where $\gamma(256)$ is the radius of a ball of unit volume in dimension $256$. In other words, the shortest vector is likely close to $0.8156\%$ larger than $2^{-1/256}\cdot \gamma(256)$, the latter being what one expects for an unstructured lattice of the same dimension.  
\end{example}
The above example might suggest to the reader that, for all intents and purposes, we may pretend the second moment $\mathbb{E}[\rho_V(\Lambda)]$ is on the nose that of a Poisson distribution of mean $V/\omega_K$ and that the error term is negligible. While this is certainly true in this example and as the blocksize increases, there is some subtle dependence on the underlying number field to take into account, especially in relatively low dimensions. For cyclotomic fields, we plot the natural logarithm of the upper bound on the normalized error 
$$\eta=\frac{1}{V\cdot\omega_K}\cdot \left(\mathbb{E}[\rho_V(\Lambda)]-V^2-V\cdot\omega_K\right)$$ 
obtained by our methods; we again computed error bounds by computationally listing $\{\alpha \in \Q(\zeta_{16})^\times: H_W(\alpha)\leq 100\}$ and using a tail estimate for the remaining points. 
\begin{figure}[ht]
\caption{Log second moment errors for cyclotomic fields $\mathbb{Q}(\zeta_m)$ for $m=8,10,12,13,15,16$ and varying ranks $15\leq t\leq 32$. }
\includegraphics[width=10cm]{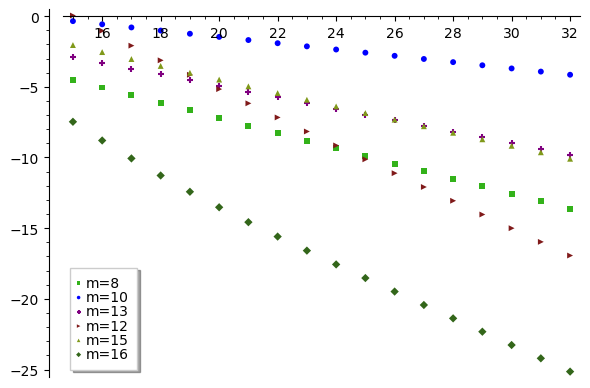}
\centering
\label{fig:momenterrors}
\end{figure}
These computations in figure \ref{fig:momenterrors} immediately generate results similar to Example \ref{ex:fixedK} and show that a conductor divisible by small odd primes leads to slightly more noise.
On the theoretical side, the sensitivity to the number field may in part be explained by the presence of units $\alpha\in\OK^\times$ of relatively small height (see the exceptions in the proof of Proposition \ref{prop:cycloheightbounds}): when the $\OK$-rank $t$ is small, the Euclidean lengths of $v,\alpha v\in\Lambda$ for any $\Lambda \in \ML_t(K)$ are now close and correlated. In future work, we hope to gain a better understanding of such questions.

\bibliographystyle{unsrt}
\bibliography{authfile}

\end{document}